\documentclass[10pt]{amsart}

\usepackage{hyperref}
\usepackage{blindtext}
\usepackage{tasks}
\usepackage{graphicx}
\usepackage{setspace}
\usepackage{marginnote}
\usepackage{amsrefs}
\usepackage{graphicx}  
\usepackage[top=3.5cm, bottom=2.5cm, outer=2.5cm, inner=2.5cm]{geometry}
\usepackage{hyperref}
\hypersetup{linkcolor=red,
citecolor=black
}
\usepackage{amsmath}

\newtheorem{thm}{Theorem}[section]
\newtheorem{lem}[thm]{Lemma}
\newtheorem{prop}[thm]{Proposition}
\newtheorem{cor}[thm]{Corollary}

\theoremstyle{definition}

\theoremstyle{remark}
\newtheorem{re}[thm]{Remark}

\numberwithin{equation}{section}

\usepackage[utf8]{inputenc}
\usepackage[english]{babel}
\usepackage{paralist}
\usepackage{amsthm}
 
\theoremstyle{definition}
\newtheorem{definition}{Definition}[section]
 
\theoremstyle{remark}

 \usepackage{bm,amsmath}

\newcommand{\tran}{^{\mathstrut\scriptscriptstyle\top}}

\newcommand{\Rey}{\mathcal{R}e }
\newcommand{\x}{\cdot}
\newcommand{\grad}{\nabla}

\allowdisplaybreaks[4]
\begin{document}

\title{Damping Functions correct over-dissipation of the Smagorinsky Model}


\author{Ali Pakzad}
\address{Department of Mathematics, University of Pittsburgh, Pittsburgh, PA 15260, USA}
\email{alp145@pitt.edu}

\maketitle
\setcounter{tocdepth}{1}

\begin{abstract}
This paper studies the time-averaged energy dissipation rate $\langle  \varepsilon_{SMD} (u)\rangle$ for the combination of the Smagorinsky model and damping function. The Smagorinsky model is well known to over-damp. One common correction is to include damping functions that reduce the effects of model viscosity near walls. Mathematical analysis is given here that allows evaluation of $\langle  \varepsilon_{SMD} (u)\rangle $ for any damping function. Moreover, the analysis motivates a modified van Driest damping. It is proven that the combination of the Smagorinsky with this modified damping function does not over dissipate and is also consistent with Kolmogorov phenomenology.
\end{abstract}

\section{introduction}
Experience with the Smagorinsky model (SM)  indicates it over dissipates (p.247 of Sagaut \cite{Sag}). This extra dissipation can laminarize the numerical approximation of a turbulent flow and prevent the transition to turbulence (p.192 of \cite{Laytonbook}). Model refinements aim at reducing model dissipation occur as early as 1975 \cite{Sch} and continues with dynamic parameter selection (Germano,  Piomelli,  Moin and Cabot \cite{GPMC} and Swierczewska \cite{A1}), structural sensors (Hughes, Oberai and Mazzei \cite{Hughes}) and near wall models (e.g.,  Piomelli and Balaras \cite{Pio}, John, Layton and Sahin \cite{JLS} and John and Liakos \cite{JL}). The classical approach is to multiply the turbulent viscosity with damping function $\beta(x)$ (such as van Driest damping \cite{Van}) with $\beta(x) \rightarrow 0$ as $x \rightarrow$ walls. There has been many numerical tests but little analytic support of this combination.  

 This paper analyzes this combination of the SM with damping function $\beta(x)$ in the flow domain $\Omega = (0,L)^3$,
\begin{equation}
  \label{Damping}
  u_t+u \cdot \nabla u -\nu \Delta u + \nabla p - \grad \x (\beta(x) (C_s \delta)^2 |\grad u| \grad u)=0\hspace{10pt} \mbox{and}\hspace{10pt} \grad \cdot u =0 \hspace{10pt} \mbox{in}\,\, \Omega.
   \end{equation}
 In (\ref{Damping})  $u$ is the velocity, $p$ is the pressure, $\nu$ is the kinematic viscosity, $\delta<<1$ is a model length scale and  $C_s \simeq 0.1$  is the standard model parameter (Lilly \cite{Lilly}). To evaluate the effect of damping function in the near wall region, we study the time-averaged energy dissipation rate of (\ref{Damping}) for shear flow. $L$-periodic boundary conditions in $x$ and $y$ directions are imposed. $z=0$ is a fixed wall and the wall $z=L$ moves with velocity $U$ (Figure \ref{Fig1}), 

\begin{equation} \label{BC}
u(x,y,0,t)=(0,0,0)\tran \hspace{18pt}\mbox{and}\hspace{18pt} u(x,y,L,t)=(U,0,0)\tran.
\end{equation}
The Reynolds number is $\Rey=\frac{U L}{\nu}$. The time-averaged energy dissipation rate for model (\ref{Damping})  includes dissipation due to the  viscous forces and turbulent diffusion reduced by the damping function $\beta(x)$. It is given by

\begin{equation}
\langle  \varepsilon_{SMD} (u)\rangle =  \limsup\limits_{T\rightarrow\infty} \frac{1}{T} \int_{0}^{T}  (\frac{1}{|\Omega|} \int_ \Omega \nu |\nabla u|^2 + (c_s\delta)^2 \beta(x) |\nabla u|^3 dx) \,\, dt.
\end{equation}

 This paper estimates  $\langle  \varepsilon_{SMD} (u)\rangle$ (Theorem \ref{thm1}) for a damping function $\beta(x) $ in terms of its integral on an $\gamma=\mathcal{O}(\Rey^{-1})$ strip along the moving wall as
$$ \langle  \varepsilon_{SMD} (u)\rangle  \leq  \big[C_1 +C_2 \, (\frac{C_s \delta}{L})^2  \, \Rey^3\,\frac{1}{L} \int_{L-\gamma L}^{L}  \beta(z)\,  dz \big] \, \frac{U^3}{L}.$$
For an algebraic approximation of van Driest damping (Section 2.2), Corollary 3.2 shows 
$$\langle  \varepsilon_{SMD} (u)\rangle \leq \big[C_2+C_2\, (\frac{C_s \delta}{L})^2   \, \Rey^2\big] \, \frac{U^3}{L}.$$
The above estimate goes to $\frac{U^3}{L}$ for fixed $\Rey$ as $\delta\rightarrow 0$, but blows up as $\Rey\rightarrow \infty$ for fixed $\delta$, suggesting over-dissipation. On the other hand, damping with a classical mixing length formula (3.19) of Prandtl, we obtain in Corollary 3.4  
$$\langle  \varepsilon_{SMD} (u)\rangle \leq C\,  \frac{U^3}{L},$$
which is consistent with Kolmogorov phenomenology.

\begin{figure}\label{Fig1}
\includegraphics[scale=0.3]{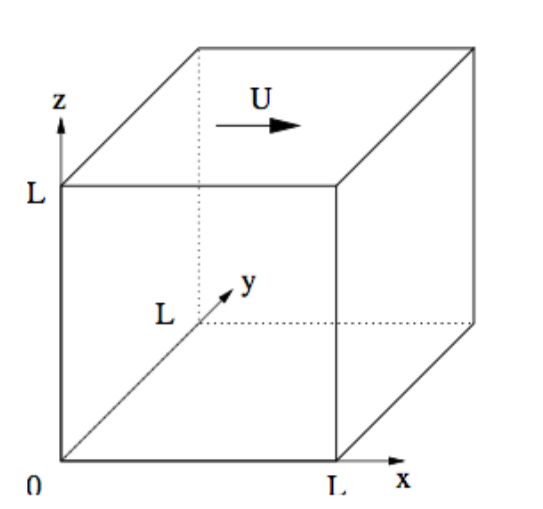}
\caption{Shear flow.}
\end{figure} 

\subsection{Related work}

 In the theory of turbulence, the time-averaged energy dissipation rate is a fundamental quantity (Sreenivasan \cite{S84}, Pope \cite{Po00}, Lesieur \cite{L97}, and Frisch \cite{F95}) determines the smallest persistent length scales and the dimension of any global attractor \cite{Layton03}. Moreover, the smallest length scale of turbulent flow simulation can be estimated by using upper bounds of the energy dissipation rate. In turbulent flows, the energy dissipation is often observed approach a limit  independent of the viscosity \cite{Kolmogorov}. No rigorous proof of this fact has been given. For shear flow between parallel plates, Busse \cite{B70} and Howard \cite{H72} estimated $\langle  \varepsilon(u) \rangle $ under the  assumptions that the flow is statistically stationary. Doering and Constantin \cite{EDR-shear} proved  an upper bound on the time-averaged energy dissipation rate for general weak solutions of the Navier-Stokes equations (NSE), $\langle  \varepsilon (u)\rangle \leq C\,  \frac{U^3}{L} $. Similar estimations have been proven by Marchiano \cite{Marchiano}, Wang \cite{Wang} and Kerswell \cite{Kerswell} in more generality.

The Smagorinsky model \cite{Smagorinsky}, $\beta(x)\equiv 1$ in (\ref{Damping}), is a common turbulence model used in Large Eddy Simulation (e.g., \cite{Layton01}, \cite{Sag}, \cite{John}, \cite{Mus}, \cite{par} and \cite{Geu}). The extra term with respect to the NSE can be generally justified as follows. In turbulence, dissipation occurs non-negligibly only at very small scales, smaller than typical mesh. The balance between energy input at the largest scales and energy dissipation at the smallest is a critical selection mechanism for determining statistics of turbulent flows. To get an accurate simulation, once a mesh is selected, an extra dissipative term must be introduced to model the effect of the unresolved fluctuations, which are smaller than the mesh width, upon the resolved velocity.

 The energy dissipation rate of the Smagorinsky model for shear flow with boundary layers was estimated in \cite{Layton03} as 
\begin{equation}
\langle  \varepsilon_s (u)\rangle \simeq [1+C_s^2 (\frac{\delta}{L})^2 (1+\Rey)^2] \frac{U^3}{L}.
\end{equation}
This estimate blows up for $\delta$ fixed as $\Rey \rightarrow \infty$, which is consistent with the numerical evidence (e.g.,  Iliescu and Fischer \cite{Fischer} and  Moin and  Kim \cite{MK}). Surprisingly, it was shown in  \cite{Layton02} that the energy dissipation rate of the Smagorinsky model  in the absence of boundary layers satisfies 

\begin{equation}
\langle  \varepsilon_s (u)\rangle \simeq  \frac{U^3}{L}.
\end{equation}
Comparing these two results (1.4) and (1.5)  suggests that  the model over dissipation is due to the action of the model viscosity in boundary layers rather than in interior small scales generated by the turbulent cascade. To reduce the effect of model viscosity in the boundary layers damping functions $\beta(x)$, which go to zero at the walls, are often used (Pope \cite{Po00}). In this case most of the tools of analysis, such as K{\"o}rn's inequality, the Poincar{\'e}-Friederichs inequality, and Sobolev's inequality, no longer hold. Thus, the mathematical development of the SM under no-slip boundary conditions with damping function is cited in \cite{Layton01} p.78 as an important open problem. 

\bigskip

\noindent{\bf Acknowledgments.} 
The author would like to thank Professor William Layton for suggesting this problem and for many fruitful discussions. A.P. was Partially supported by NSF grants DMS 1522267 and CBET 1609120. 

\section{Mathematical preliminaries}
We use the standard notations $L^p(\Omega), W^{k,p}(\Omega), H^k(\Omega)= W^{k,2}(\Omega)$ for the Lebesgue and Sobolev spaces respectively. The inner product in the space $L^2(\Omega)$ will be denoted by $(\cdot,\cdot)$ and its norm by $ || \x || $ for scalar, vector and tensor quantities. Norms in Sobolev spaces $H^k(\Omega), k>0$, are denoted by  $ || \x ||_k$ and the usual $L^p$ norm is denoted by $|| \x ||_{L^p}$. The symbols $C$ and $C_i$ for $i=1, 2, 3$ stand for generic positive constant independent of the $\nu$, $L$ and $U$. $\grad u$ is the gradient tensor $(\grad u)_{ij} =\frac{\partial u_j}{\partial x_i}$ for $i, j= 1, 2 , 3$.

\begin{definition}

The velocity at a given time $t$ is sought in the space

$\mathbb{X}(\Omega):= \{u \in H^1(\Omega):u(x,y,0)=(0,0,0)\tran, \, u(x,y,L)=(U,0,0)\tran, \, \mbox{$u$ is $L$-periodic in $x$ and $y$ direction} \}.$

The test function space is 

$\mathbb{X}_0(\Omega):= \{u \in H^1(\Omega):u(x,y,0)=(0,0,0)\tran,\, u(x,y,L)=(0,0,0)\tran,\, \mbox{$u$ is $L$-periodic in $x$ and $y$ direction} \}.$

The pressure at time $t$ is sought in 

${Q}(\Omega):= L_0^2(\Omega) = \{q \in L^2(\Omega): \hspace{3pt}\int_{\Omega} q dx= 0\}.$

And the space of divergence-free functions is denoted by

$V(\Omega):=\{ u \in \mathbb{X}(\Omega):  \hspace{3pt}(\grad \x u,q)=0  \hspace{5pt} \forall  q \in {Q}\}.$
\end{definition}

\begin{definition} 
\textbf{(Trilinear from)} Define  $b: \mathbb{X}\times \mathbb{X} \times \mathbb{X} \rightarrow \mathbb{R}$ as $b(u,v,w):= (u \x \grad v, w) $.
\end{definition}

\begin{lem} \label{trilinear}
The nonlinear term  $b(\x,\x,\x)$ is continuous on $\mathbb{X}\times \mathbb{X} \times \mathbb{X}$ (and thus on $V \times V \times V$ as well). Moreover,  we have the following skew-symmetry property for $b$ 
$$b(u,v,v)=0 \hspace{10pt} \forall u\in V , v \in \mathbb{X}.$$
\end{lem}
\begin{proof}
The proof is standard and the one with zero boundary conditions can be found in  p.114 of   Girault and Raviart \cite{Raviart}.
\end{proof}

\subsection{Construction of background flow} 
  
One key step to the upper bound on $\langle  \varepsilon_{SMD} (u)\rangle $ is to construct an appropriate background flow, $\Phi \in \mathbb{X}(\Omega)$, following Hopf \cite{Hopf} and Doering and Constantin \cite{EDR-shear}. This is a divergence-free function extending the boundary condition (\ref{BC}) to the interior of $\Omega$. Moreover, $u-\Phi \in \mathbb{X}_0(\Omega)$ and will be used as a test function in the weak form (\ref{Damping-weak}). The choice of $\gamma \in (0,1)$ will be determined by the needs of the estimates in (\ref{choiceofgamma}) and it will be chosen to be $\gamma = \frac{1}{5.1}\, (\Rey)^{-1}$.

\begin{definition}\label{backgroundflow}
 \textbf{(The background flow)} Define $\Phi(x,y,z) := (\phi(z),0,0)\tran$, where
 
$$
\phi(z) =
\left\{
	\begin{array}{ll}
		0  & \mbox{if } z\in [0,L-\gamma L]  \\
		\frac{U}{\gamma L} (z-(L-\gamma L))  & \mbox{if }  z\in [L-\gamma L, L]
	\end{array}.
\right.
$$
\end{definition}
$\phi(z)$ is sketched in Figure 2. We collect two properties for $\Phi$ in Lemmas \ref{lemma1} and \ref{lemma2}.

\begin{lem} \label{lemma1}
 $\Phi$ satisfies
  
  \begin{tasks}(4)
\task ${\Vert\Phi\Vert}_{\infty} \leq U,$
\task ${\Vert\nabla\Phi\Vert}_{\infty} \leq \frac{U}{\gamma L},$
\task ${\Vert \Phi\Vert}^2  \leq\frac{U^2 \gamma L^3}{3},$
\task $ \Vert \nabla \Phi\Vert^ 2 \leq \frac{U^2 L}{\gamma}.$
\end{tasks}
 \end{lem}
  
    \begin{proof}
  
 They all are the immediate consequence of the Definition \ref{backgroundflow}. We  show ($c$) here as an example.\\
 
   ${\Vert \Phi\Vert}^2 = L^2 \int_{0}^{L} |\phi (z)| ^2  dz = L^2 \int_{L-\gamma L}^{L} \frac{U^2}{(\gamma L)^ 2} (z-(L-\gamma L))^2 dz =\frac{U^2 \gamma L^3}{3}.$\\
  \end{proof}

  \begin{figure}
\includegraphics[scale=0.3]{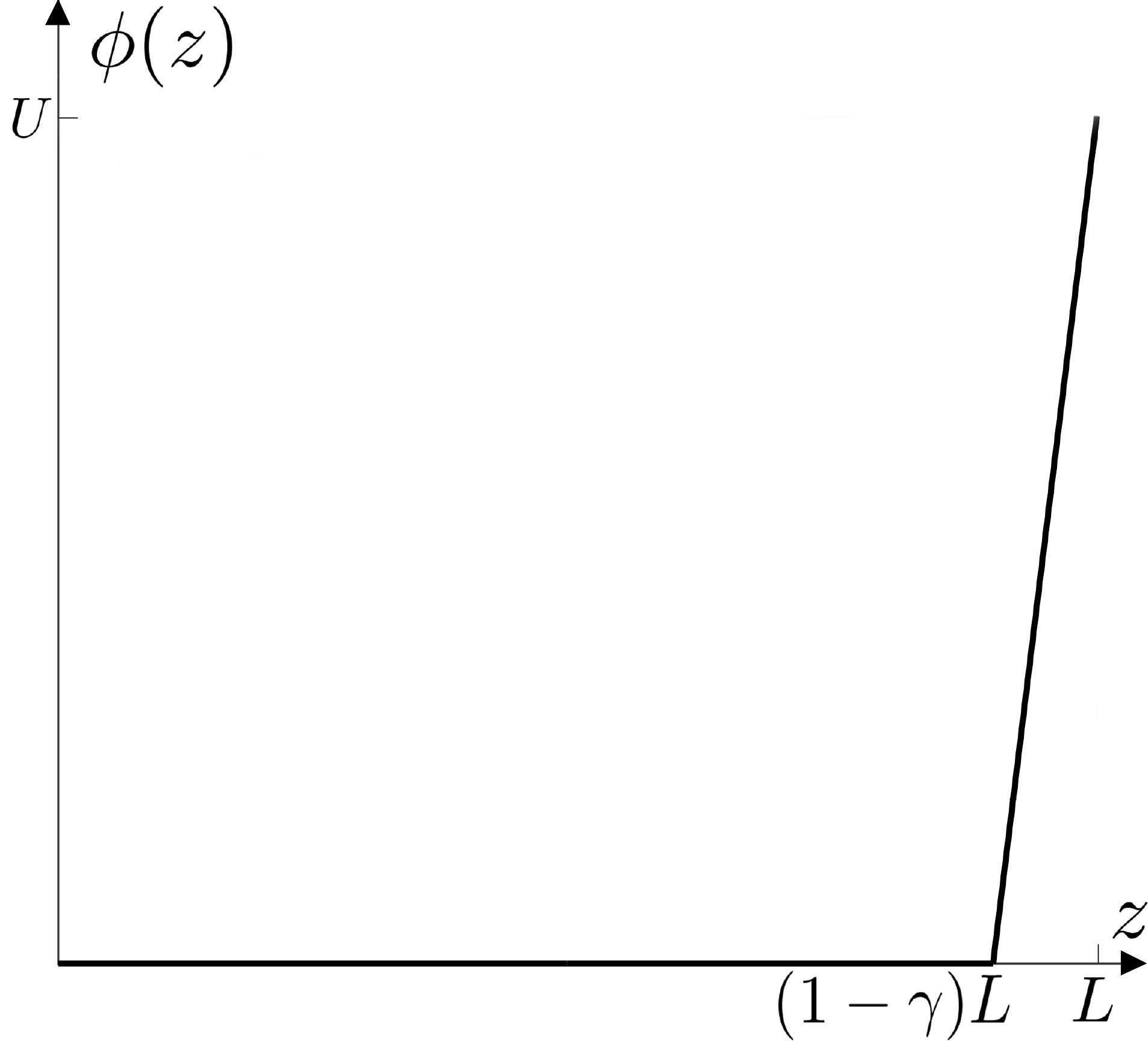}
\caption{The background flow.}
\end{figure}

We will need the well-known dependence of the Poincar\'e -Friedrichs inequality constant on the domain. A straightforward argument in the thin domain $\mathcal{O}_{\gamma L}$ implies Lemma \ref{lemma2}.

 \begin{lem} \label{lemma2}
 Let $\mathcal{O}_{\gamma L}=\lbrace(x,y,z)\in \Omega : L-\gamma L \leq z \leq L\rbrace  $ 
  be the region close to the upper boundary. Then we have
  \begin{equation}
   {\Vert u - \Phi\Vert}_{L^2(\mathcal{O}_{\gamma L})} \leq \gamma L  {\Vert\nabla (u -\Phi)\Vert }_{L^2(\mathcal{O}_{\gamma L})}.
    \end{equation}
  \end{lem}
 \begin{proof}
 
 First  let $v $ be a $C^1$ function on $ \mathcal{O}_{\gamma L}$ that vanishes for $z=L$.  Then component-wise $ (i=1,2,3)$, we have
  $$ v_{i}(x,y,z) = v_{i} (x,y,L) -\int_{z}^{L} \frac{d   v_{i}}{d\xi} (x,y,\xi) d\xi.$$
  Observing that $v_{i}(x,y,L) = 0 $, squaring both sides, and using the Cauchy-Schwarz inequality, we get 
  $$ v^2_{i}(x,y,z)\leq \gamma L \int _{L-\gamma L}^{L} (\frac{d v_i}{d\xi}(x,y,\xi))^2 d\xi.$$
 Integrating both sides with respect to $z$ gives
  $$\int_{L-\gamma L}^{L} v^2_{i}(x,y,z) dz \leq (\gamma L)^2 \int _{L-\gamma L}^{L} (\frac{d v_i}{d\xi}(x,y,\xi))^2 d\xi.$$  
Then integrating with respect to $x $ and $ y$ and summing from $i=1$ to 3, we obtain
  $$\Vert v\Vert^2_ {L^2(\mathcal{O}_{\gamma L})}\leq (\gamma L)^2 \Vert\nabla v\Vert^2_{L^2(\mathcal{O}_{\gamma L})}.$$
This proves the lemma for  $ v \in C^1 $. Finally use a density argument and take $v = u - \Phi $.
 \end{proof}
 \subsection{The kinetic energy}
 Before proving the main theorem, we prove boundedness of the kinetic energy, $\frac{1}{2} ||u||^2$, and that $\langle  \varepsilon_{SMD} (u)\rangle$ is well-defined. The proof of the model (\ref{Damping}) is similar to the NSE case first presented  in Hopf \cite{Hopf}. We need the following proposition first.
 \begin{prop}\label{prop1}
Let $\grad v \in L^p(\mathcal{O}_{\gamma L})$ and $0<p<\infty$. If $v(x,y,L,t)=0$ then 
$$||\frac{v(x,y,z)}{L-z}||_{L^p(\mathcal{O}_{\gamma L})}\leq \frac{p}{p-1}\, ||\frac{\partial v}{\partial z} (x,y,z)||_{{L^p(\mathcal{O}_{\gamma L})}}.$$
\end{prop}
\begin{proof}
Using B. Hardy's inequality (P.313 of Brezis \cite{Brezis}) when $z \in [L-\gamma L, L]$ for fixed $x$ and $y$ gives 

$$||\frac{v(x,y,z)}{L-z}||_{L^p([L-\gamma L, L])}\leq \frac{p}{p-1} ||\frac{\partial v}{\partial z} (x,y,z)||_{{L^p([L-\gamma L, L])}}.$$
Raising  both sides to power $p$, then taking a double integral with respect to $x$ and $y$ for $x, y \in [0, L]$ implies the result,

$$||\frac{v(x,y,z)}{L-z}||^p_{L^p(\mathcal{O}_{\gamma L})}\leq (\frac{p}{p-1})^p \, ||\frac{\partial v}{\partial z} (x,y,z)||^p_{{L^p(\mathcal{O}_{\gamma L})}} \leq (\frac{p}{p-1})^p \,||\grad v||^p_{{L^p(\mathcal{O}_{\gamma L})}}.$$ 
\end{proof}

\begin{lem}\label{Lemmma0}
   The kinetic energy and the time averages of the energy dissipation of the solution to (\ref{Damping}) with the boundary conditions (\ref{BC}) are  uniformly bounded in time, i.e.
$$\sup_{t \in (0,\infty)} ||u(t)|| \leq C < \infty\hspace{15pt} \mbox{and}\hspace{15pt} \sup_{t \in (0,\infty)} \frac{1}{T} \int_{0}^{T}  (\frac{1}{|\Omega|} \int_ \Omega \nu |\nabla u|^2 + (c_s\delta)^2 \beta(x) |\nabla u|^3 dx) \,\, dt \leq C < \infty.$$
\end{lem}
\begin{proof}
The strategy is to subtract off the inhomogeneous boundary conditions  (\ref{BC}).  Consider  $v= u-\Phi$, then $v$ satisfies homogeneous boundary conditions. Substituting $u=v+\Phi$ in the equation (\ref{Damping}) yields 
\begin{equation}
  v_t+v \cdot \nabla v -\nu \Delta v + \nabla p+\phi(z) \frac{\partial v}{\partial x} + v_3 \phi'(z) (1,0,0)- \grad \x (\beta(x) (C_s \delta)^2 |\grad (v+\Phi)| \grad (v+\Phi))=0,
   \end{equation}
   $$ \grad \cdot v =0.$$
with boundary conditions
\begin{equation} 
v(x,y,0,t)=(0,0,0)\tran \hspace{18pt}\mbox{and}\hspace{18pt} v(x,y,L,t)=(0,0,0)\tran, 
\end{equation}
and 
\begin{equation}
v(x+L,y,z,t)=v(x,y,z,t)\hspace{18pt}\mbox{and} \hspace{18pt}v(x,y+L,z,t)=v(x,y,z,t).
\end{equation}

Taking inner product with $v=(v_1,v_2,v_3)$ and integrating over $\Omega$ give 
\begin{equation}
\frac{1}{2} \frac{\partial}{\partial t} ||v||^2 + \nu ||\grad v||^2 + \int_{\Omega} \big(\phi(z) \frac{\partial v}{\partial x}\x v + v_1 v_3 \phi'(z)\big) \, dx + (\beta(x) (C_s \delta)^2 |\grad (v+\Phi)| \grad (v+\Phi), \grad v)=0.
\end{equation}
Since any integral containing $\phi$ and $\phi'$ will be zero outside the strip $\mathcal{O}_{\gamma L}$, by integrating by part we have

$$\int_{\Omega} \phi(z) \frac{\partial v}{\partial x}\x v \,dx = \frac{1}{2} \int_{\mathcal{O}_{\gamma L}} \phi(z) \frac{\partial}{\partial x} |v|^2\, dx= \frac{1}{2} \int_{\mathcal{O}_{\gamma L}} \frac{\partial}{\partial x}(\phi(z) |v|^2) \,dx -\frac{1}{2} \int_{\mathcal{O}_{\gamma L}}\phi'(z) |v|^2 \,dx = -\frac{1}{2} \frac{U}{\gamma L}  \int_{\mathcal{O}_{\gamma L}} |v|^2 \,dx.$$
Inserting this identity in (2.5) and using the triangle inequality on the last term give

\begin{equation}
\begin{split}
 &\frac{\partial}{\partial t} ||v||^2 + 2 \nu ||\grad v||^2 +\frac{U}{\gamma L}  \int_{\mathcal{O}_{\gamma L}}\big(2 v_1 v_3 -  |v|^2 \big)\,dx + 2\int_{\Omega} \beta(x) (C_s \delta)^2 |\grad v|^3 \,dx\\
 &\leq  4\int_{\Omega}(\beta(x) (C_s \delta)^2|\grad v|^2\, |\grad \Phi|\,dx +2\int_{\Omega}(\beta(x) (C_s \delta)^2 |\grad v|\, |\grad \Phi|^2 \, dx.
 \end{split}
\end{equation}

The rest of analysis requires to approximate various term in the above. Let $p=2$ in Proposition \ref{prop1} and the two terms on the LHS can be bounded above as 

\begin{equation}
\begin{split}
\frac{U}{\gamma L} \int_{\mathcal{O}_{\gamma L}} |v|^2 \,dx &= \frac{U}{\gamma L} \int_{\mathcal{O}_{\gamma L}} d(z)^2 |\frac{v}{d(z)}|^2 \,dx \leq \frac{U}{\gamma L}(\gamma L)^2 \int_{\mathcal{O}_{\gamma L}}  |\frac{v}{d(z)}|^2 \,dx\\
&\leq 2 U \gamma L \int_{\mathcal{O}_{\gamma L}} |\grad v|^2\, dx \leq 2U\gamma L ||\grad v||^2.
\end{split}
\end{equation}
Similarly 
\begin{equation}
\frac{U}{\gamma L} \int_{\mathcal{O}_{\gamma L}}2 |v_1 v_3| \,dx\leq 4U\gamma L ||\grad v||^2.
\end{equation}
To bound the two terms on the RHS of (2.6), use H{\"o}lder's inequality and Young inequality $||fg||_{L^1}\leq ||f||_{L^p}\,||g||_{L^q}\leq\frac{\epsilon}{p}|f||_{L^p}^p+\frac{\epsilon^{-\frac{q}{p}}}{q}||g||_{L^q}^q $. Consider the first term, for  $p=\frac{3}{2}$, $q=3$ and $\epsilon=0.6$ we have 
\begin{equation}
\begin{split}
4\int_{\Omega}(\beta(x) (C_s \delta)^2|\grad v|^2\, |\grad \Phi|\,dx &\leq 4(C_s \delta)^2 \,\big(\int_{\Omega} \beta(x) |\grad v|^3\, dx\big)^\frac{2}{3}\, \big(\int_{\Omega} \beta(x) |\grad \Phi|^3\, dx\big)^\frac{1}{3}\\
&\leq 1.6 \int_{\Omega} \beta(x) (C_s \delta)^2|\grad v|^3\, dx+ 4 \int_{\Omega} \beta(x) (C_s \delta)^2 |\grad \Phi|^3\, dx.
\end{split}
\end{equation}
The second term is estimated exactly like the last term for $p=3$, $q=\frac{3}{2}$ and $\epsilon=0.6$ as

\begin{equation}
\begin{split}
2\int_{\Omega}(\beta(x) (C_s \delta)^2 |\grad v|\, |\grad \Phi|^2 \, dx
&\leq 2(C_s \delta)^2 \,\big(\int_{\Omega} \beta(x) |\grad v|^3\, dx\big)^\frac{1}{3}\, \big(\int_{\Omega} \beta(x) |\grad \Phi|^3\, dx\big)^\frac{2}{3}\\
&\leq 0.4 \int_{\Omega} \beta(x) (C_s \delta)^2|\grad v|^3\, dx+ 2 \int_{\Omega} \beta(x) (C_s \delta)^2 |\grad \Phi|^3\, dx.
\end{split}
\end{equation}
Inserting these last four estimates into the energy inequality (2.6) for $v$ gives
$$\frac{\partial}{\partial t} ||v||^2+(2 \nu- 6 U \gamma L) ||\grad v||^2\leq 6 \int_{\Omega} \beta(x) (C_s \delta)^2 |\grad \Phi|^3\, dx.$$
Thus, if $\gamma$ is chosen small enough that 
$$\gamma < \frac{1}{3}(\Rey)^{-1},$$
then $(2 \nu- 6 U \gamma L)$ becomes positive. Applying the Poincar{\'e}-Friedrichs inequality $||v||\leq C  ||\grad v||$ gives 

$$\frac{\partial}{\partial t} ||v||^2+C || v||^2\leq 6 \int_{\Omega} \beta(x) (C_s \delta)^2 |\grad \Phi|^3\, dx.$$
Since RHS is uniformly bounded in time, a standard Grönwall's inequality shows that 
$$\sup_{t \in (0,\infty)} ||v(t)|| \leq C < \infty,$$
Which proves the boundedness of the kinetic energy, $\frac{1}{2}||u||^2$. From this and standard arguments it follows that
$$\frac{1}{T} \int_{0}^{T}  (\frac{1}{|\Omega|} \int_ \Omega \nu |\nabla u|^2 + (c_s\delta)^2 \beta(x) |\nabla u|^3 dx) \,\, dt \leq C < \infty,$$
which means $\langle  \varepsilon_{SMD} (u)\rangle$ is well-defined.
\end{proof}

 \subsection{van Driest damping}
 To modify the mixing-length model van Driest proposed \cite{Van}, with some theoretical support but mainly as a good fit to data (p.77 of Wilcox \cite{Wilcox}), that the mixing length $\ell$ should be multiplied by the damping function so that $\ell (x) \rightarrow 0$ as $x\rightarrow$ wall. The van Driest damping function is 
\begin{equation}
f_w(z)=1- e^{\frac{-z^+}{A^+}} , 
\end{equation}
where $A^+ =26$ is the van Driest constant and $z^+$ is the non-dimensional distance from the wall (p.76 of Wilcox \cite{Wilcox})
\begin{equation}
z^+= \frac{u_{\tau}(L-z)}{\nu} , 
\end{equation}
which determines the relative importance of viscous and turbulent phenomena. $u_{\tau}$ is the wall shear velocity given by 
\begin{equation}
u_{\tau} = \sqrt{\left. \nu \frac{\partial u}{\partial z} \right|_{\mbox{wall}}}.
\end{equation}
$u_{\tau}$ is still unknown, the analysis herein will require a specific value for $u_{\tau}$. To this end, it can be estimated as follows. Near the wall $\grad u \simeq \frac{\partial u}{\partial z}$, then  

\begin{equation}
u_{\tau} = \sqrt{\left. \nu \frac{\partial u}{\partial z} \right|_{\mbox{wall}}} \simeq \sqrt{\left. \nu \,\, \grad u \right|_{\mbox{wall}}} = \sqrt[4]{ \nu ^2 (\left. \grad u \right|_{\mbox{wall}})^2} \simeq \sqrt[4]{ \nu \langle \bar{\epsilon}_w \rangle},
\end{equation}
where $ \bar{\epsilon}_w $ is a spatial-average energy dissipation rate near the wall. After assuming a non-zero fraction occurs in near-wall region and therefore  neglecting the effects of viscosity far from the boundary layer, dissipation occurs mainly in the boundary layers near the bottom and top walls which both have a volume of $ L^3\,\gamma$. Hence 

$$ \langle \epsilon \rangle =  2 \frac{1}{L^3} \, ( L^3\,\gamma)  \langle \bar{\epsilon}_w \rangle= 2 \gamma \langle \bar{\epsilon}_w \rangle,$$
On the other hand, based on the statistical equilibrium $ \langle \epsilon \rangle = \frac{U^3}{L}$, therefore

$$ \langle \epsilon \rangle = \frac{U^3}{L} = 2 \gamma \langle \bar{\epsilon}_w \rangle,$$
Using $\gamma=\mathcal{O}(\Rey^{-1})$ gives

\begin{equation}
\langle \bar{\epsilon}_w \rangle \simeq \frac{1}{2} \frac{U^4}{\nu}.
\end{equation}
Then  $u_{\tau}$ is estimated by inserting (2.15) in (2.14) to be
 \begin{equation}
 u_{\tau}\simeq \frac{U}{\sqrt[4]{2}}.
 \end{equation}
Hence van Driest damping function is approximated as (Figure 3)
\begin{equation}
f_w(z)\simeq 1- \exp (\frac{-U (L-z)}{ 26 \sqrt[4]{2} \,\, \nu}).
\end{equation}
Using Taylor series to approximate (2.17) in the boundary layer ${\mathcal{O}_{\gamma L}} $  gives 

\begin{equation}
f_w(z)\simeq \sum\limits_{n=1}^{k} [\frac{1}{26\, \sqrt[4]{2}\, n!} \,\Rey \,  (1- \frac{z}{L})]^n + \mathcal{O}(\Rey ^{k+1}\, (1-\frac{z}{L})^{k+1}).
\end{equation}

\begin{figure}[!tbp]
  \centering
  \begin{minipage}[b]{0.4\textwidth}
    \includegraphics[width=\textwidth]{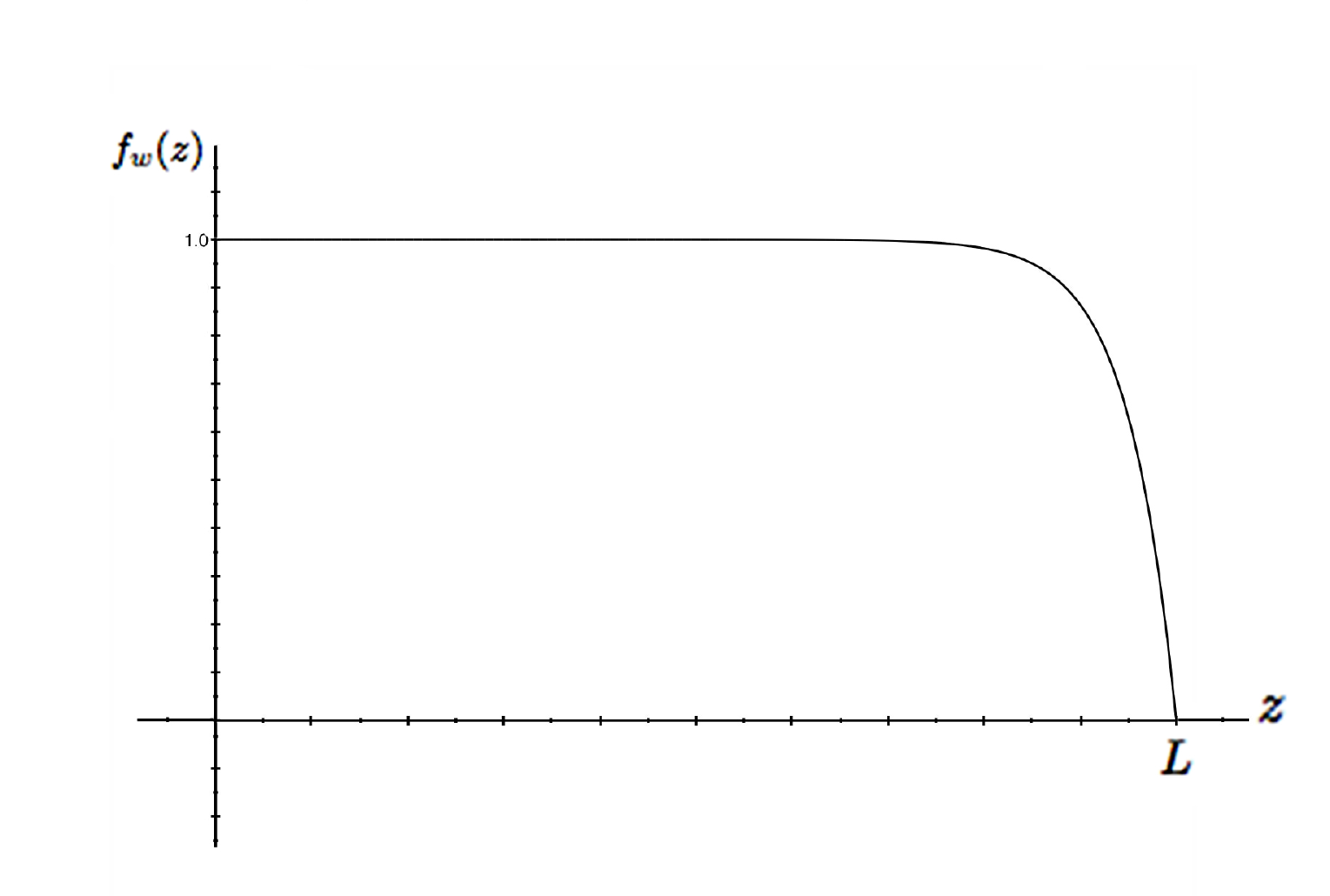}
    \caption{van Driest.}
  \end{minipage}
  \hfill
  \begin{minipage}[b]{0.4\textwidth}
    \includegraphics[width=\textwidth]{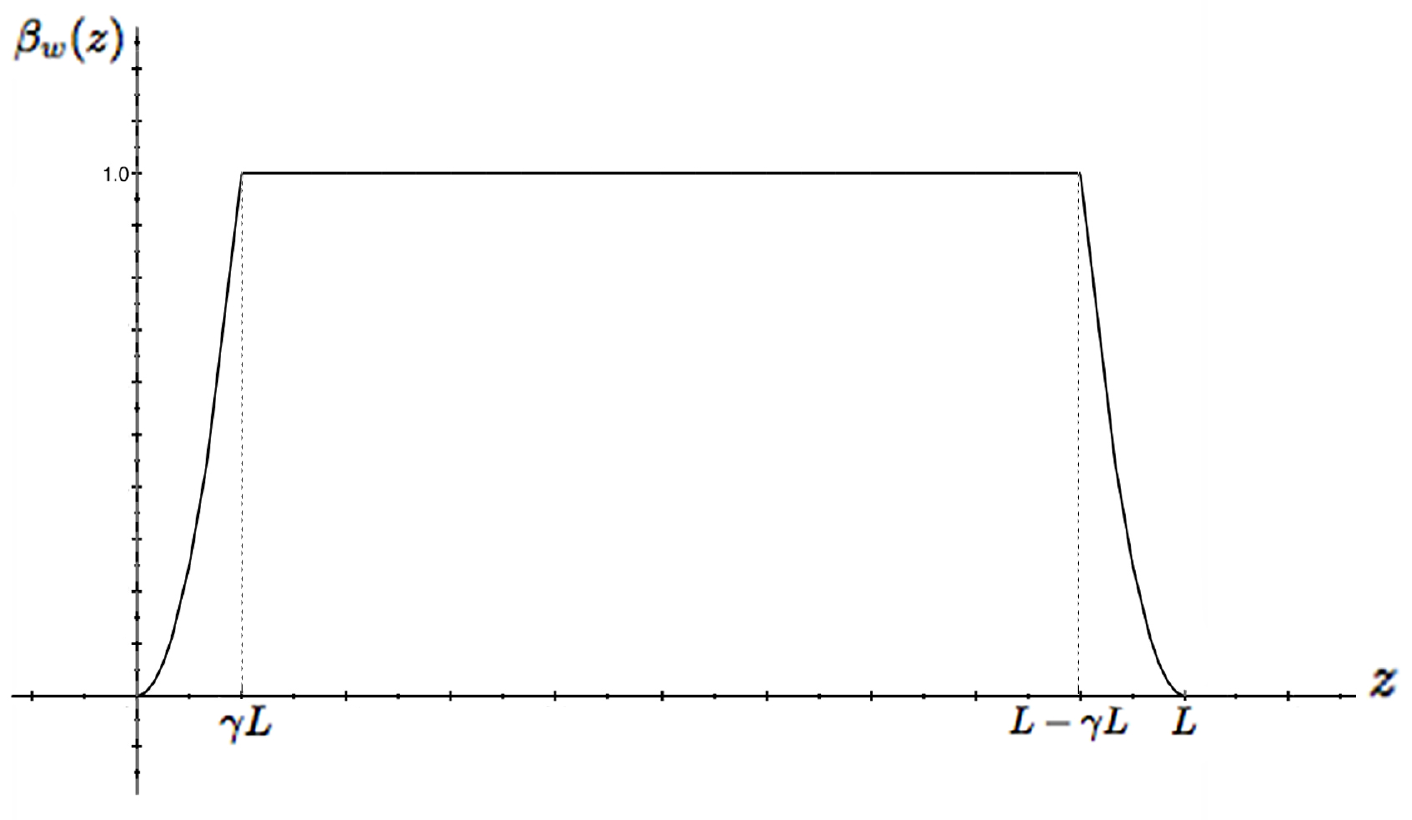}
    \caption{Algebraic approximation.}
  \end{minipage}
\end{figure}

Note that the above approximation (2.18) is valid when the reminder $\Rey \, (1-\frac{z}{L})$ is less than 1, and this occurs when $L-\gamma L \leq  z \leq L.$ Moreover,  approximation (2.18) suggests $(\Rey)^\alpha (1-\frac{z}{L})^\alpha$ for $\alpha\geq1$ as a damping function only on the top layer. Thus (2.19) is an algebraic approximation to the van Driest damping on the whole domain (Figure 4, for $\alpha =2$).
 
 \begin{equation}
 \beta_w(z) =
\left\{
	\begin{array}{ll}
		(\Rey)^\alpha (\frac{z}{L})^\alpha & \mbox{if } z\in [0,\gamma L]  \\
		1  & \mbox{if }  z\in [\gamma L, L-\gamma L]\\
		(\Rey)^\alpha (1-\frac{z}{L})^\alpha & \mbox{if } z\in [L-\gamma L, L] 
	\end{array}.
\right.
 \end{equation}

 \begin{re}
 $\beta_w$ plays the role of $\beta$ in (\ref{Damping}).
 \end{re}
 
 \section{Analysis of the Smagorinsky with Damping Function}
 
\begin{thm}\label{thm1}
Suppose $u_0 \in L^2(\Omega)$ and let \footnote{In fact $\gamma$ can be $ \kappa (\frac{1}{5} \Rey^{-1})$ for any $\kappa \in (0,1)$. Without loss of generality, $\gamma$ is taken to be  $\frac{1}{5.1} \Rey^{-1} $ for simplicity in calculations.} $\gamma =\frac{1}{5.1}\, \Rey^{-1}$ , then for any positive damping function $\beta(z) \in L^1(\Omega)$, $\langle  \varepsilon_{SMD} (u)\rangle$ satisfies 
$$ \langle  \varepsilon_{SMD} (u)\rangle  \leq \big[ C_1 +C_2 \, (\frac{C_s \delta}{L})^2  \, \Rey^3\,\frac{1}{L} \int_{L-\gamma L}^{L}  \beta(z)\,  dz \big]\, \frac{U^3}{L}.$$
\end{thm}
\begin{proof}

The weak form of (\ref{Damping}) is obtained by taking the scalar product $ v\in\mathbb{X}_0$ and $q \in L_0^2 $ with (\ref{Damping}) and integrating over the space $\Omega$. 

\begin{equation}
  \label{Damping-weak} 
  \begin{aligned}
 (u_t, v) + \nu (\nabla u,\nabla v) + b(u,u,v) - (p,\nabla \x v)+(\beta (x)(C_s \delta)^2 |\grad u| \grad u,\grad v)&= 0 & \forall v \in \mathbb{X}_0  ,
    \\
  (\nabla \x u,q) & =0 & \forall q\in L_0^2  ,
   \\
    (u(x,0) - u_0(x) , v)
    &=
    0 &\forall v \in \mathbb{X}_0.  
   \end{aligned}
  \end{equation}
  Take $v = u- \Phi$ in (\ref{Damping-weak}). Using the skew-symmetry of  $b(\x,\x,\x)$ (Lemma \ref{trilinear}) and $\grad \x \Phi =0$ gives

$$(u_t,u-\Phi)+\nu(\nabla u,\nabla u-\nabla\Phi)+b(u,u,u-\Phi)+(\beta (x)(C_s \delta)^2|\nabla u| \nabla u, \nabla u-\nabla\Phi)=0,$$
which is equivalent to the following
\begin{equation} 
\begin{split}
\frac{1}{2}\frac{d}{dt}\Vert u\Vert^2 +\nu \Vert \nabla u\Vert ^2+(C_s \delta)^2(\beta(x) |\nabla u|\nabla u,\nabla u) & = \frac{d}{dt}(u,\Phi)+b(u,u,\Phi)+\nu(\nabla u,\nabla\Phi)\\
& + (C_s \delta)^2(\beta (x)|\nabla u|\nabla u,\nabla \Phi).
\end{split}
\end{equation}
Integrating with respect to time from above equation gives
\begin{equation} 
\begin{split}
\frac{1}{2} \Vert u(T)\Vert^2 - \frac{1}{2} \Vert u(0)\Vert^2 &+\nu \int _{0}^{T} \Vert \nabla u\Vert^2 dt +\int_{0}^{T}(\int_\Omega (C_s \delta)^2 \beta (x) |\nabla u|^3 dx) dt = (u(T),\Phi)-(u(0),\Phi) \\
 & +\int _{0}^{T}b(u,u,\Phi) dt+\nu \int _{0}^{T} (\nabla u,\nabla\Phi) dt
 + (C_s \delta)^2 \int_{0}^{T}(\beta (x)|\nabla u| \nabla u, \nabla \Phi) dt.\\
\end{split}
\end{equation}

 The proof continues by bounding, term by term, each term on the right-hand side of the energy equality (3.3). Using the Cauchy-Schwarz Young inequality and Lemma \ref{lemma1}, the first three terms are estimated as follows.

\begin{equation}
(u(T),\Phi)\leq \frac{1}{2}\Vert u(T)\Vert^2 +\frac{1}{2}\Vert\Phi\Vert^2 = \frac{1}{2}\Vert u(T)\Vert^2 +\frac{U ^2 \gamma L^3}{6}.
\end{equation}
\begin{equation}
(u(0),\Phi)\leq \Vert u(0)\Vert \Vert\Phi\Vert = \sqrt{\frac{\gamma}{3}} U L^{\frac{3}{2}} \Vert u(0)\Vert.
\end{equation}
\begin{equation}
\nu \int _{0}^{T} (\nabla u,\nabla\Phi) dt \leq \frac{\nu}{2} \int _{0}^{T}\Vert \nabla u\Vert^2 + \Vert \nabla\Phi\Vert ^2  dt = \frac{\nu}{2} \int _{0}^{T}\Vert \nabla u\Vert^2 dt + \frac{\nu}{2}  \frac{U^2 L}{\gamma}T.
\end{equation}

For the nonlinear term $b(\x,\x,\x)$ in (3.3) add and subtract terms and  then use skew-symmetry. This gives 

\begin{equation} \label{eq1}
\begin{split}
b(u,u,\Phi) & = b(u-\Phi,u-\Phi,\Phi)+ b(\Phi,u-\Phi,\Phi)  \\
 & = \frac{1}{2}  b(u-\Phi,u-\Phi,\Phi) - \frac{1}{2} b(u-\Phi,\Phi,u-\Phi)\\
 & +\frac{1}{2}  b(\Phi,u-\Phi,\Phi) - \frac{1}{2}  b(\Phi,\Phi,u-\Phi).
\end{split}
\end{equation}
To estimate the four terms in (3.7), use  Lemma \ref{lemma1}, Lemma \ref{lemma2}, the Cauchy-Schwarz Young inequality. Moreover, apply the fact that  $b(u,u,\Phi)$ is an integration on $\mathcal{O}_{\gamma L}$ since  supp$(\Phi) =\overline{\mathcal{O}} _{\gamma L}$. For the first term in (3.7) we have 
\begin{equation} 
\begin{aligned}
& b(u-\Phi,u-\Phi,\Phi) \leq \Vert\Phi\Vert_{L^\infty(\mathcal{O}_{\gamma L})} \Vert u-\Phi\Vert_{L^2(\mathcal{O}_{\gamma L})}  \Vert \nabla(u-\Phi)\Vert_{L^2(\mathcal{O}_{\gamma L})}\leq\gamma LU \Vert \nabla (u-\Phi)\Vert^2_{L^2(\mathcal{O}_{\gamma L})}\\
  & \leq\gamma LU \Vert \nabla u- \nabla\Phi\Vert^2_{L^2} \leq UL\gamma (\Vert\nabla u\Vert+ \Vert\nabla\Phi\Vert)^2  \leq UL\gamma (2 \Vert \nabla u \Vert^2 + 2 \Vert \nabla \Phi \Vert^2) \\
 &\leq UL\gamma (2 \Vert \nabla u \Vert^2+2 \frac{U^2 L}{\gamma})= 2 UL\gamma \Vert \nabla u\Vert^2 +2 U^3L^2.\\
 \end{aligned}
\end{equation}
For the second term we have
\begin{equation} 
\begin{aligned}
&b(u-\Phi,\Phi,u-\Phi)  \leq \Vert \nabla\Phi\Vert_{L_\infty(\mathcal{O}_{\gamma L})} \Vert u-\Phi\Vert_{L^2(\mathcal{O}_{\gamma L})} ^2 \leq \frac{U}{\gamma L} \gamma ^2 L^2 \Vert \nabla(u-\Phi)\Vert_{L^2(\mathcal{O}_{\gamma L})}^2\\
& \leq \gamma ^2 L^2 \frac{U}{\gamma L}(2 \Vert \nabla u \Vert^2+2 \frac{U^2 L}{\gamma})=2 \gamma LU \Vert \nabla u\Vert^2+ 2U^3 L^2.
\end{aligned}
\end{equation} 
The third one is estimated as
\begin{equation} 
\begin{split}
&b(\Phi,u-\Phi,\Phi)  \leq \Vert \Phi\Vert_{L^\infty(\mathcal{O}_{\gamma L})} \Vert \nabla (u-\Phi)\Vert_{L^2(\mathcal{O}_{\gamma L})} \Vert \Phi\Vert_{L^2(\mathcal{O}_{\gamma L})} \leq U \sqrt{\frac{U^2\gamma L^3}{3}}(\Vert \nabla u\Vert+ \Vert \nabla \Phi\Vert)\\
& \leq U \sqrt{\frac{U^2\gamma L^3}{3}}(\Vert \nabla u\Vert+ \sqrt{\frac{U^2 L}{\gamma}})  \leq\frac{U^2\gamma ^ {\frac{1}{2}} L^{\frac{3}{2}}}{\sqrt{3}} \Vert \nabla u\Vert + \frac{U^3 L^2}{\sqrt{3}}\\
 & =[ \frac{U^{\frac{3}{2}} L}{\sqrt{3}}]\,\,\, [(U\gamma L)^{\frac{1}{2}}\Vert \nabla u\Vert] + \frac{U^3 L^2}{\sqrt{3}} \leq (\frac{U^3 L^2}{6}) + \frac{1}{2} UL\gamma \Vert \nabla u\Vert ^2+ \frac{U^3 L^2}{\sqrt{3}} \\
 &= \frac{1}{2} UL\gamma \Vert \nabla u\Vert^2 + (\frac{\sqrt{3}}{3}+\frac{1}{6})   U^3 L^2. \\
 \end{split}
\end{equation}
And finally the last one satisfies 
\begin{equation} 
\begin{split}
&b(\Phi,\Phi,u-\Phi) \leq \Vert\Phi\Vert_{L^\infty(\mathcal{O}_{\gamma L})} \Vert \nabla\Phi\Vert_{L^2(\mathcal{O}_{\gamma L})} \Vert u-\Phi\Vert_{L^2(\mathcal{O}_{\gamma L})}  \leq U\sqrt{\frac{U^2 L}{\gamma}} \gamma L \Vert \nabla (u-\Phi)\Vert_{L^2(\mathcal{O}_{\gamma L})}\\
&\leq U^2 \gamma^{\frac{1}{2}} L^{\frac{3}{2}} (\Vert\nabla u\Vert +\Vert\nabla\Phi\Vert)\leq U^2 \gamma^{\frac{1}{2}} L^{\frac{3}{2}} (\Vert\nabla u\Vert +(\frac{U^2 L}{\gamma})^{\frac{1}{2}})\\
&= U^2 \gamma^{\frac{1}{2}} L^{\frac{3}{2}} \Vert\nabla u\Vert + U^3 L^2 =[U^{\frac{3}{2}}L]\,\,[(UL\gamma)^{\frac{1}{2}} \Vert \nabla u\Vert]+ U^3 L^2\\
&\leq\frac{1}{2}(U^3 L^2)+ \frac{1}{2} UL\gamma \Vert \nabla u\Vert^2 + U^3 L^2 =\frac{1}{2} UL\gamma \Vert \nabla u\Vert^2+\frac{3}{2}U^3 L^2.\\
\end{split}
\end{equation}
Using (3.8), (3.9), (3.10) and (3.11) in (3.7) gives the final estimation for the non-linear term as below.

\begin{equation}
 | b(u,u,\Phi)|\leq \frac{5}{2}UL\gamma \Vert \nabla u\Vert^2 +\frac{19}{6} U^3 L^2.
\end{equation}

Finally the last term on the RHS of (3.3) can be estimated as the follows. Using  Hölder's inequality and Young inequality for $p=\frac{3}{2}$ and $q=3$  gives 

\begin{equation} 
\begin{split}
|(\beta(x)|\nabla u| \nabla u, \nabla \Phi)|&  \leq \int_\Omega |\beta(x)| |\nabla u|^2 |\nabla \Phi| dx \\
 &=\int_\Omega (\beta ^{\frac{2}{3}}|\nabla u|^2)\,\,(|\beta| ^{\frac{1}{3}}|\nabla \Phi|) dx\\
 & \leq\ [\int_\Omega \beta | \nabla u|^3 dx]^{\frac{2}{3}}\,\,[\int_\Omega \beta | \nabla \Phi|^3 dx]^{\frac{1}{3}}\leq\ \frac{2}{3} \int_\Omega\beta | \nabla u|^3 dx+ \frac{1}{3}\int_\Omega \beta | \nabla \Phi|^3 dx.
\end{split} 
\end{equation}
Inserting (3.4), (3.5), (3.6), (3.12) and (3.13) in (3.3) implies 
\begin{equation}
\begin{split}
 \frac{1}{2} \Vert u(T)\Vert ^2 -& \frac{1}{2} \Vert u(0)\Vert ^2 +\nu \int_{0}^{T} \Vert \nabla u\Vert ^2 dt +(C_s \delta)^2 \int _{0}^{T} (\int _\Omega  \beta (x)|\nabla u|^3 dx)dt  \leq  \frac{1}{2}\Vert u(T)\Vert^2 +\frac{U ^2 \gamma L^3}{6} \\
 &+ \sqrt{\frac{\gamma}{3}} U L^{\frac{3}{2}} \Vert u(0)\Vert  + \frac{5}{2}UL\gamma \int_{0}^{T}\Vert \nabla u\Vert^2 dt+\frac{19}{6} U^3 L^2 T +  \frac{\nu}{2} \int _{0}^{T}\Vert \nabla u\Vert^2 dt + \frac{\nu}{2}  \frac{U^2 L}{\gamma} T \\
 & + \frac{2}{3} (C_s \delta)^2  \int_{0}^{T}(\int_\Omega \beta(x) |\nabla u|^3 dx)dt + \frac{1}{3} (C_s \delta)^2\int_{0}^{T}(\int_\Omega \beta(x) |\grad \Phi|^3 dx) dt.
\end{split}
\end{equation}
Since the kinetic energy is bounded (Lemma \ref{Lemmma0}), the above inequality becomes
\begin{equation} 
\begin{split}
(\frac{1}{2}- \frac{5}{2 \nu}& \gamma LU) \int_{0}^{T} \nu \Vert \nabla u\Vert ^2 dt +\frac{1}{3} (C_s \delta)^2 \int _{0}^{T} (\int _\Omega  \beta (x)|\nabla u|^3 dx)dt \leq  \frac{1}{2} \Vert u(0)\Vert ^2+ \frac{1}{6} U^2\gamma L^3    \\
 &+ \sqrt{\frac{\gamma}{3}} U L^{\frac{3}{2}} \Vert u(0)\Vert + \frac{19}{6}U^3 L^2 T + \frac{\nu}{2 \gamma} L U^2 T +\frac{1}{3} (C_s \delta)^2 \int_{0}^{T}(\int_\Omega \beta(x) |\grad \Phi|^3 dx) dt.
\end{split}
\end{equation}
Finally dividing both sides of (3.15) by $T$ and $ |\Omega| = L^3$ and  taking the limit superior leads to
\begin{equation}\label{choiceofgamma}
  min \lbrace \frac{1}{2} - \frac{5}{2} \frac{\gamma L U}{\nu}, \frac{1}{3}\rbrace \langle  \varepsilon_{SMD} (u)\rangle  \leq \frac{19}{6} \frac{U^3}{L} +\frac{\nu}{2\gamma} \frac{U^2}{L^2} + \frac{1}{3}\frac{1}{L^3} (C_s \delta)^2 \int_\Omega \beta(x) |\grad \Phi|^3 dx.
 \end{equation}
 The above inequality leads to the last step when $C_1$ and $C_2$ are positive and independent of viscosity, diam($\Omega$) and lid velocity. Take $\gamma = \frac{1}{5.1}\Rey^{-1}$, then $(\frac{1}{2}- \frac{5}{2 \nu} \gamma LU) $ becomes positive and therefore

\begin{equation}
 \langle  \varepsilon_{SMD} (u)\rangle  \leq C_1 \frac{U^3}{L}+C_2 \frac{1}{L^3} (C_s \delta)^2 \int_\Omega \beta(x) |\grad \Phi|^3 dx.
\end{equation}
Because the background flow $\Phi$ vanishes on $(\Omega \backslash \mathcal{O}_{\gamma L} )$, we have
\begin{equation}
\int_{\Omega} \beta(x) |\grad \Phi|^3 dx= (\frac{U}{\gamma L})^3\int_{0}^{L} \int_{0}^{L}\int_{L- \gamma L}^{ L} \beta (x,y,z) dx dy dz = \frac{U^3}{\gamma^3 L} \int_{L-\gamma L}^{L} \beta(z) dz.
\end{equation}
Inserting (3.18) in (3.17) proves Theorem \ref{thm1}.
\end{proof}

\subsection{Evaluation of Damping Functions}
  
  Theorem \ref{thm1} is the starting point for the evaluation of damping functions. It is next applied to two damping functions in Corollaries \ref{Cor1} and \ref{Cor2} and the result compared. 

\begin{cor}\label{Cor1}
For the algebraic approximation of the van Driest damping function, $\beta_w(z)$ in (2.10), we have 
$$ \langle  \varepsilon_{SMD} (u)\rangle  \leq  \big[ C_1 +C_2 \,\frac{1}{\alpha+1}\, (\frac{C_s \delta}{L})^2   \, \Rey^2 \big] \, \frac{U^3}{L}.$$
\end{cor}
 \begin{proof}
The result is a calculation by applying $\beta_w(z)$ in (2.10) to the Theorem \ref{thm1}.
 \end{proof}
 The upper bound in  Corollary \ref{Cor1} is a function of the global velocity $U$, domain diameter $L$, the eddy size $\delta$, and surprisingly, the Reynolds number. Moreover, it blows up as $\Rey \rightarrow \infty$. Due to this fact one can propose the following modification to $\beta_w(z)$. Consider $\beta_d(z) \in C^1(\Omega)$  in (3.19) which is based on a connection of the algebraic damping near the wall smoothly to the no damping in the interior by hermite interpolation. It is given by
 
 \begin{equation}
 \beta_d(z) =
\left\{
	\begin{array}{ll}
		 (\frac{z}{L})^\alpha (1-\frac{z}{L})^\alpha & \mbox{if } z\in [0,\gamma L]  \\
		 a_1(z-\gamma L)^3+b_1(z-\gamma L)^2+c_1(z-\gamma L)+d_1 & \mbox{if }  z\in [\gamma L,  2 \gamma L]\\
		1  & \mbox{if }  z\in [2 \gamma L, L- 2\gamma L]\\
		 a_2 (z+ 2\gamma L-L)^3 +b_2 (z+2\gamma L -L)^2+1 & \mbox{if }  z\in [L- 2\gamma L,  L- \gamma L]\\
	   (\frac{z}{L})^\alpha (1-\frac{z}{L})^\alpha & \mbox{if } z\in [L-\gamma L, L] 
	\end{array},
\right.
 \end{equation}

where $\alpha\geq 0$ and $a_1, a_2, b_1, b_2, c_1$ and $d_1$ are constant such that

\begin{itemize}
\item $a_1= \frac{-2}{\gamma^3 L^3}+\frac{1}{L^3} \alpha \gamma^{\alpha-3} (1-\gamma)^{\alpha-1}(1-2\gamma)+\frac{2}{L^3} \gamma^{\alpha-3} (1-\gamma)^\alpha$,
\item  $b_1= \frac{3}{\gamma^2 L^2} -\frac{2}{L^2} \alpha \gamma^{\alpha-2} (1-\gamma)^{\alpha-1} (1-2 \gamma)-\frac{3}{L^2} \gamma^{\alpha-2} (1-\gamma)^\alpha$,
\item   $c_1= \frac{1}{L} \alpha \gamma^{\alpha-1} (1-\gamma)^{\alpha-1}(1-2\gamma)$, 
\item $d_1= \gamma^\alpha (1-\gamma)^\alpha$,
\item $a_2= - a_1$,
\item $b_2= -\frac{3}{\gamma^2 L^2} +\frac{1}{L^2} \alpha \gamma^{\alpha-2} (1-\gamma)^{\alpha-1} (1-2 \gamma)+\frac{3}{L^2} \gamma^{\alpha-2} (1-\gamma)^\alpha$.
\end{itemize}

 \begin{figure}
 \includegraphics[scale=0.4]{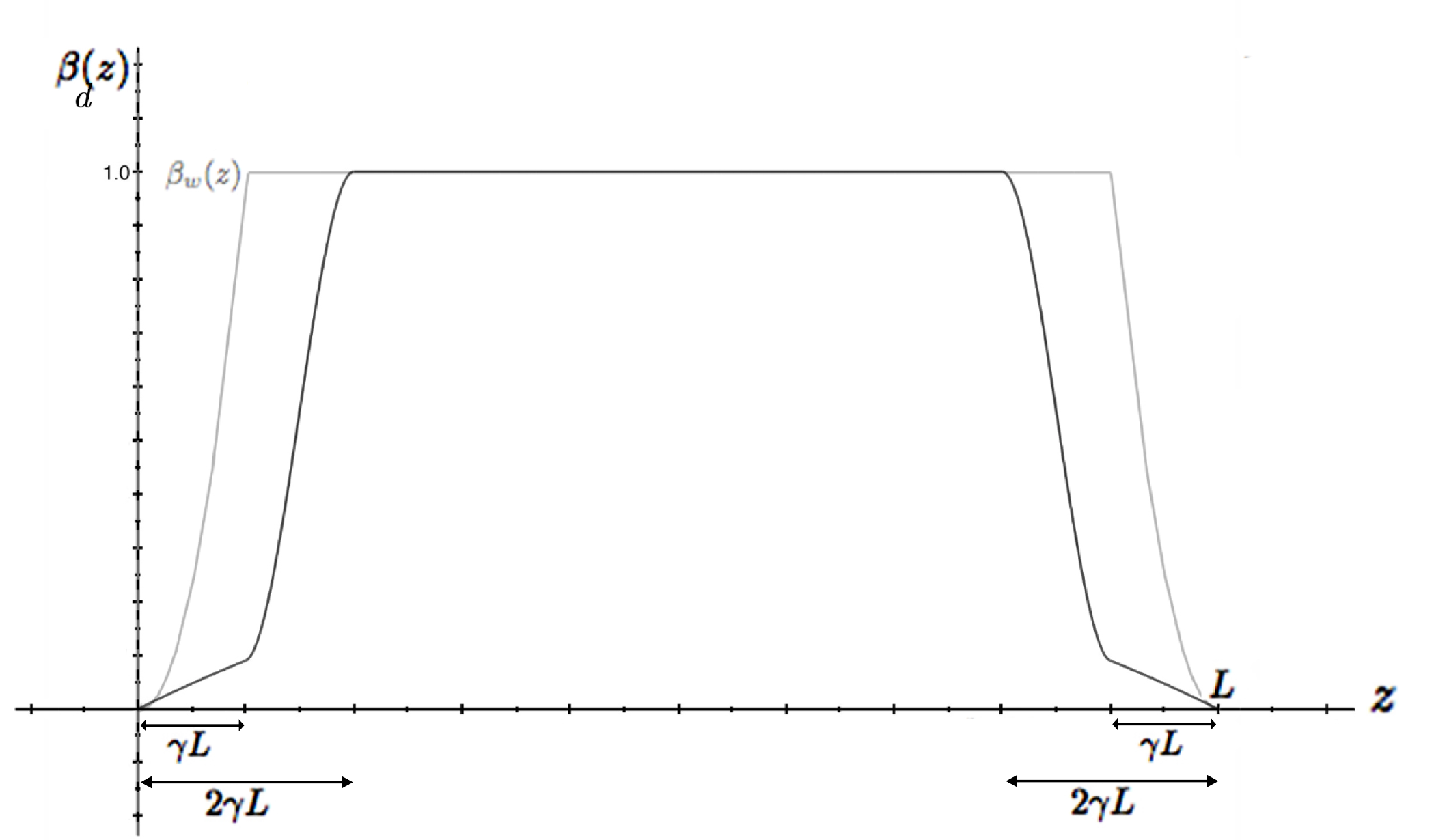}
\caption{Proposed damping function (3.19).}
\end{figure}

  \begin{re}
All the constants above are calculated such that $\beta_d(z) \in C^1$. $\beta_d$, which plays the role of $\beta$ in (\ref{Damping}), is sketched in Figure 5 and compared to $\beta_w$. They both are symmetric, bounded and vanish at  $z=0$ and $z=L$. Moreover, they are almost 1 on the whole domain except on the thin boundary layers. 
  \end{re}
  
 \begin{cor}\label{Cor2}
Suppose $u_0 \in L^2(\Omega)$ and $\beta_d$ given by (3.19) with $\,2 \leq \alpha \in\mathbb{N} $. Then, for any $\Rey\geq 1$ we have 

$$ \langle  \varepsilon_{SMD} (u)\rangle  \leq \big[ C_1+C_2(\frac{C_s \delta}{L})^2 \big]  \,\, \frac{U^3}{L}.$$

\end{cor}
\begin{proof}
Considering $\beta_d$ is symmetric on the whole domain implies 
$$
\int_{L- \gamma L}^{ L} \beta_d (z)  dz =  \int_{0}^{\gamma L} (\frac{z}{L})^ {\alpha} (1-\frac{z}{L})^{\alpha} dz.
$$
Now applying the Binomial Theorem on $(1- \frac{z}{L})^\alpha$ and then taking integral gives
\begin{equation}
\begin{split}
\int_{L- \gamma L}^{ L} \beta_d (z)  dz= L \gamma ^{\alpha+1}\big( &\frac{1}{\alpha+1} -\frac{1}{\alpha+2} \alpha \gamma + \frac{1}{\alpha+3} \frac{\alpha (\alpha-1)}{2} \gamma ^2\\
&+ ...+ (-1)^{\alpha} \frac{1}{2\alpha} \alpha \gamma^{\alpha-1} +(-1)^{\alpha}\frac{1}{2 \alpha+1} \gamma^{\alpha}\big).
\end{split}
\end{equation}
After dropping negative terms, since $\gamma = \frac{1}{5.1} (\Rey)^{-1} \ll1$ the RHS of (3.20) can be bounded above by a constant, $C_{\alpha}$, which depends on $\alpha$. Therefore
\begin{equation}
\int_{L- \gamma L}^{ L} \beta_d (z)  dz\leq C_{\alpha} L \gamma^{\alpha+1},
\end{equation}
Using  $\gamma = \frac{1}{5.1} (\Rey)^{-1}$  and inserting the above in the Theorem \ref{thm1} imply  
$$ \langle  \varepsilon_{SMD} (u)\rangle  \leq \big[ C_1+C_2\,(\frac{C_s \delta}{L})^2 \,(\frac{1}{5.1})^3\, C_{\alpha} \,\gamma ^{\alpha-2} \big]  \,\, \frac{U^3}{L}.$$
Use the assumptions $\alpha\geq 2$ and $\gamma = \frac{1}{5.1} (\Rey)^{-1} \ll1$ imply $\gamma^{\alpha-2}\leq 1$ and now the corollary is proved.  

 \end{proof}

Corollary 3.4 is in accordance with the Kolmogorov theory of turbulence. It establishes that the combination of SM with damping function $\beta_d(z) $  given by (3.19) does not over dissipate, and the energy input rate $\frac{U^3}{L}$ is balanced by $\langle  \varepsilon _{SMD}(u)\rangle $.  This estimate is consistent with the rate proven for the NSE in \cite{EDR-shear} and \cite{Foias-charles}; it is also dimensionally consistent. 

 The assumption $\alpha\geq 2$ is a significant one in the analysis. When $\alpha=1$ we obtain the following corollary.
\begin{cor}\label{Cor3}
Suppose $\alpha=1$ in Corollary 3.4, then 
$$ \langle  \varepsilon_{SMD} (u)\rangle  \leq  \big[ C_1+C_2(\frac{C_s \delta}{L})^2 \,  \Rey \big] \, \frac{U^3}{L}.$$

\end{cor}
\begin{proof}
The proof follows that of $\alpha \geq 2$ in the Corollary 3.4  except the inequality (3.21) is modified to 
 $$\int_{L-\gamma L}^{L} \beta_d(z)\, dz \leq  C\, L\,  (\Rey)^{-2},
  $$
  for $\alpha=1.$
\end{proof}

\section{Conclusion}
The key parameter is  $\alpha=$ the order of contact of the damping function at the wall. Comparing $\langle  \varepsilon_{SMD} (u)\rangle  \simeq \frac{U^3}{L}$ for $\alpha\geq 2$ in Corollary \ref{Cor2} with $\langle  \varepsilon_{SMD} (u)\rangle \simeq \big[ C_1+C_2(\frac{C_s \delta}{L})^2 \,  \Rey \big] \, \frac{U^3}{L} $ for $\alpha=1$ in Corollary \ref{Cor3} suggests that the model over dissipates flows for $\alpha=1$. If the upper bounds are sharp (an open problem), the accurate simulation would need $\alpha \geq 2$. The next logical step is to study $\langle  \varepsilon_{SMD} (u^h)\rangle$ after discretization by fixed mesh $h$ and extend the results in this paper, specially when the mesh does not resolve the boundary layers.

\end{document}